\documentclass [reqno]{amsart}

\usepackage{amsmath,amsfonts,amssymb,amscd,verbatim,delarray,fullpage}

\usepackage{stmaryrd}
\usepackage{amsthm}
\usepackage{url}
\usepackage[polutonikogreek,english]{babel}

\usepackage{pgf}
\usepackage{latexsym,amsmath,amssymb}
\usepackage[parfill]{parskip} 

\DeclareMathOperator{\frob}{Frob}

\DeclareMathOperator{\tr}{tr}

\chardef\bslash=`\\ 

\begin{document}


\newtheorem{Theorem}{Theorem}[section]

\newtheorem{cor}[Theorem]{Corollary}

\newtheorem{Conjecture}[Theorem]{Conjecture}

\newtheorem{Lemma}[Theorem]{Lemma}
\newtheorem{lemma}[Theorem]{Lemma}
\newtheorem{property}[Theorem]{Property}
\newtheorem{proposition}[Theorem]{Proposition}
\newtheorem{ax}[Theorem]{Axiom}
\newtheorem{claim}[Theorem]{Claim}

\newtheorem{nTheorem}{Surjectivity Theorem}

\theoremstyle{definition}
\newtheorem{Definition}[Theorem]{Definition}
\newtheorem{problem}[Theorem]{Problem}
\newtheorem{question}[Theorem]{Question}
\newtheorem{Example}[Theorem]{Example}

\newtheorem{remark}[Theorem]{Remark}
\newtheorem{diagram}{Diagram}
\newtheorem{Remark}[Theorem]{Remark}
\newcommand{\diagref}[1]{diagram~\ref{#1}}
\newcommand{\thmref}[1]{Theorem~\ref{#1}}
\newcommand{\secref}[1]{Section~\ref{#1}}
\newcommand{\subsecref}[1]{Subsection~\ref{#1}}
\newcommand{\lemref}[1]{Lemma~\ref{#1}}
\newcommand{\corref}[1]{Corollary~\ref{#1}}
\newcommand{\exampref}[1]{Example~\ref{#1}}
\newcommand{\remarkref}[1]{Remark~\ref{#1}}
\newcommand{\corlref}[1]{Corollary~\ref{#1}}
\newcommand{\claimref}[1]{Claim~\ref{#1}}
\newcommand{\defnref}[1]{Definition~\ref{#1}}
\newcommand{\propref}[1]{Proposition~\ref{#1}}
\newcommand{\prref}[1]{Property~\ref{#1}}
\newcommand{\itemref}[1]{(\ref{#1})}


\newcommand{\CE}{\mathcal{E}}
\newcommand{\CG}{\mathcal{G}}\newcommand{\CV}{\mathcal{V}}
\newcommand{\CL}{\mathcal{L}}
\newcommand{\CM}{\mathcal{M}}
\newcommand{\A}{\mathcal{A}}
\newcommand{\CO}{\mathcal{O}}
\newcommand{\B}{\mathcal{B}}
\newcommand{\CS}{\mathcal{S}}
\newcommand{\CX}{\mathcal{X}}
\newcommand{\CY}{\mathcal{Y}}
\newcommand{\CT}{\mathcal{T}}
\newcommand{\CW}{\mathcal{W}}
\newcommand{\CJ}{\mathcal{J}}

\newcommand{\st}{\sigma}
\renewcommand{\k}{\varkappa}
\newcommand{\Frac}{\mbox{Frac}}
\newcommand{\XC}{\mathcal{X}}
\newcommand{\wt}{\widetilde}
\newcommand{\wh}{\widehat}
\newcommand{\mk}{\medskip}
\renewcommand{\sectionmark}[1]{}
\renewcommand{\Im}{\operatorname{Im}}
\renewcommand{\Re}{\operatorname{Re}}
\newcommand{\la}{\langle}
\newcommand{\ra}{\rangle}
\newcommand{\LND}{\mbox{LND}}
\newcommand{\Pic}{\mbox{Pic}}
\newcommand{\lnd}{\mbox{lnd}}
\newcommand{\GLND}{\mbox{GLND}}\newcommand{\glnd}{\mbox{glnd}}
\newcommand{\Der}{\mbox{DER}}\newcommand{\DER}{\mbox{DER}}
\renewcommand{\th}{\theta}
\newcommand{\ve}{\varepsilon}
\newcommand{\1}{^{-1}}
\newcommand{\iy}{\infty}
\newcommand{\iintl}{\iint\limits}
\newcommand{\capl}{\operatornamewithlimits{\bigcap}\limits}
\newcommand{\cupl}{\operatornamewithlimits{\bigcup}\limits}
\newcommand{\suml}{\sum\limits}
\newcommand{\ord}{\operatorname{ord}}
\newcommand{\gal}{\operatorname{Gal}}
\newcommand{\bk}{\bigskip}
\newcommand{\fc}{\frac}
\newcommand{\g}{\gamma}
\newcommand{\be}{\beta}
\newcommand{\dl}{\delta}
\newcommand{\Dl}{\Delta}
\newcommand{\lm}{\lambda}
\newcommand{\Lm}{\Lambda}
\newcommand{\om}{\omega}
\newcommand{\ov}{\overline}
\newcommand{\vp}{\varphi}
\newcommand{\kap}{\varkappa}

\newcommand{\Vp}{\Phi}
\newcommand{\Varphi}{\Phi}
\newcommand{\BC}{\mathbb{C}}
\newcommand{\C}{\mathbb{C}}\newcommand{\BP}{\mathbb{P}}
\newcommand{\BQ}{\mathbb {Q}}
\newcommand{\BM}{\mathbb{M}}
\newcommand{\BR}{\mathbb{R}}\newcommand{\BN}{\mathbb{N}}
\newcommand{\BZ}{\mathbb{Z}}\newcommand{\BF}{\mathbb{F}}
\newcommand{\BA}{\mathbb {A}}
\renewcommand{\Im}{\operatorname{Im}}
\newcommand{\idd}{\operatorname{id}}
\newcommand{\ep}{\epsilon}
\newcommand{\tp}{\tilde\partial}
\newcommand{\doe}{\overset{\text{def}}{=}}
\newcommand{\supp} {\operatorname{supp}}
\newcommand{\loc} {\operatorname{loc}}
\newcommand{\de}{\partial}
\newcommand{\z}{\zeta}
\renewcommand{\a}{\alpha}
\newcommand{\G}{\Gamma}
\newcommand{\der}{\mbox{DER}}

\newcommand{\Spec}{\operatorname{Spec}}
\newcommand{\Sym}{\operatorname{Sym}}
\newcommand{\Aut}{\operatorname{Aut}}

\newcommand{\Idd}{\operatorname{Id}}

\newcommand{\tG}{\widetilde G}

\newcommand{\FX}{\mathfrac {X}}
\newcommand{\FV}{\mathfrac {V}}
\newcommand{\SX}{\mathcal {X}}
\newcommand{\SV}{\mathcal {V}}
\newcommand{\SO}{\mathcal {O}}
\newcommand{\SD}{\mathcal {D}}
\newcommand{\Sr}{\rho}
\newcommand{\SR}{\mathcal {R}}

\setcounter{equation}{0} \setcounter{section}{0}

\newcommand{\ds}{\displaystyle}
\newcommand{\gl}{\lambda}
\newcommand{\gL}{\Lambda}
\newcommand{\gge}{\epsilon}
\newcommand{\gG}{\Gamma}
\newcommand{\ga}{\alpha}
\newcommand{\gb}{\beta}
\newcommand{\gd}{\delta}
\newcommand{\gD}{\Delta}
\newcommand{\gs}{\sigma}
\newcommand{\mbq}{\mathbb{Q}}
\newcommand{\mbr}{\mathbb{R}}
\newcommand{\mbz}{\mathbb{Z}}
\newcommand{\mbc}{\mathbb{C}}
\newcommand{\mbn}{\mathbb{N}}
\newcommand{\mbp}{\mathbb{P}}
\newcommand{\mbf}{\mathbb{F}}
\newcommand{\mbe}{\mathbb{E}}
\newcommand{\lcm}{\text{lcm}\,}
\newcommand{\mf}[1]{\mathfrak{#1}}
\newcommand{\ol}[1]{\overline{#1}}
\newcommand{\mc}[1]{\mathcal{#1}}
\newcommand{\nequiv}{\equiv\hspace{-.13in}/\;}

\title{Elliptic aliquot cycles of fixed length}
\author{Nathan Jones}
\date{}

\begin{abstract}
Silverman and Stange define the notion of an aliquot cycle of length $L$ for a fixed elliptic curve $E$ over 
$\mbq$, and conjecture an order of magnitude for the function which counts such aliquot cycles.  In the present 
note, we combine heuristics of Lang-Trotter with those of Koblitz to refine their conjecture to a precise 
asymptotic formula by specifying the appropriate constant.  We give a criterion for positivity of the 
conjectural constant, as well as some numerical evidence for our conjecture.
\end{abstract}

\maketitle

\section{Introduction} \label{introduction}

Let $E$ be an elliptic curve over $\mbq$ and fix a positive integer $L \geq 2$.  In analogy with the classical notion of an aliquot cycle, Silverman and Stange \cite{silvermanstange} define an $L$-tuple $(p_1, p_2, \dots, p_L)$ of distinct positive integers to be an \textbf{aliquot cycle of length $L$ for $E$} if each $p_i$ is a prime number of good reduction for $E$ and
\begin{equation*} 
p_{i+1} = | E(\mbf_{p_i}) | \quad \forall i \in \{1, 2, \dots, L-1 \} \quad \text{ and } \quad p_1 = | E(\mbf_{p_L}) |,
\end{equation*}
which may be more succinctly written as
\begin{equation} \label{aliquotdef}
p_{i+1} = | E(\mbf_{p_i}) |, \quad \forall i \in \mbz/L\mbz.
\end{equation}
When $L = 2$, an aliquot cycle is also referred to as an \textbf{amicable pair for $E$}.  As observed in
\cite[Remark 1.5]{silvermanstange}, there is an intimate connection between aliquot cycles for $E$ and elliptic
divisibility sequences, which relate to generalizations of classical index divisibility questions about Lucas
sequences (see also \cite{gottschlich}, which studies some distributional aspects of elliptic divisibility sequences).  

It is of interest to know how common such aliquot cycles are, so we presently consider the
function which counts aliquot cycles of fixed length for a fixed elliptic curve $E$ over $\mbq$.  More precisely, 
define an aliquot cycle 
$(p_1, p_2, \dots, p_L)$ to be \textbf{normalized} if $p_1 = \min \{ p_i : 1 \leq i \leq L \}$, and 
then write
\[
\pi_{E,L}(x) := | \{ p_1 \leq x : \text{ $\exists$ a normalized aliquot cycle $(p_1,p_2, \dots, p_L)$ for $E$} \} |.
\]
The behavior of $\pi_{E,L}(x)$ for large $x$ depends heavily on whether or not
$E$ has complex multiplication (CM), as the following conjecture indicates.
\begin{Conjecture} \label{silvermanstangeconjecture}
(Silverman-Stange)  Let $E$ be an elliptic curve over $\mbq$ and $L \geq 2$ a fixed integer, and assume that 
there are infinitely many primes $p$ such that $| E(\mbf_p) |$ is prime.  Then, as $x \rightarrow \infty$, one 
has
\[
\pi_{E,L}(x) \quad
\begin{cases}
 \asymp \frac{\sqrt{x}}{(\log x)^L} & \text{ if $E$ has no CM} \\
 \sim A_{E} \frac{x}{(\log x)^2} & \text{ if $E$ has CM and $L = 2$}, \\
\end{cases}
\]
where the implied constants in $\asymp$ are both positive and depend only on $E$ and $L$, and $A_{E}$ is a 
positive constant.
\end{Conjecture}

\begin{Remark}
We may interpret the $L = 1$ case of \eqref{aliquotdef} as describing primes $p_1$ for which 
$p_1 = |E(\mbf_{p_1})|$.  Such primes are called \textbf{anomalous} primes and have been considered by 
Mazur \cite{mazur}.  The asymptotic count for anomalous primes up to $x$ is a special case of a conjecture 
of Lang and Trotter \cite{langtrotter}.
\end{Remark}

In \cite{silvermanstange}, Silverman and Stange focus on the intricacies of the CM case, proving that 
if $E$ has CM, $j_E \neq 0$ and $L \geq 3$, then $E$ any normalized aliquot cycle $(p_1,p_2,\dots, p_L)$ for 
$E$ must have
$p_1 < 5$ (so in particular, $\pi_{E,L}(x) = O(1)$).  The case $j_E = 0$ is apparently more complicated, and 
no proof is given that $\pi_{E,L}(x) = O(1)$ when $j_E = 0$ and $L > 3$.

In this note, we refine Conjecture \ref{silvermanstangeconjecture} to an asymptotic formula in the non-CM case.  
Heuristics will be developed which lead to the following conjecture.
\begin{Conjecture} \label{mainconjecture}
Let $E$ be an elliptic curve over $\mbq$ without complex multiplication and $L \geq 2$ a fixed integer.  Then 
there is a non-negative real constant $C_{E,L} \geq 0$ (see \eqref{Casinverselimit} below) so 
that, as $x \longrightarrow \infty$,
\begin{equation*} 
\pi_{E,L}(x) \sim 
C_{E,L} \int_2^x \frac{1}{2\sqrt{t} (\log t)^L} dt.
\end{equation*}
\end{Conjecture}
\begin{Remark}
It is possible for the constant $C_{E,L}$ to be zero, in which case $\ds \lim_{x \rightarrow \infty} \pi_{E,L}(x)$ 
is provably finite.  Thus, in case $C_{E,L} = 0$, let us interpret the above asymptotic to mean that 
$\ds \lim_{x \rightarrow \infty} \pi_{E,L}(x) < \infty$.
\end{Remark}
\begin{Remark}
By integration by parts, one has
\[
\int_2^x \frac{1}{2\sqrt{t} (\log t)^L} dt \; = \; \frac{\sqrt{x}}{(\log x)^L} + 
O\left( \frac{\sqrt{x}}{(\log x)^{L+1}} \right).
\]
Thus, Conjecture \ref{mainconjecture} is consistent with Conjecture \ref{silvermanstangeconjecture}. 
In practice, the error term $\ds \left| \pi_{E,L}(x) - C_{E,L} \int_2^x \frac{1}{2\sqrt{t} (\log t)^L} dt \right|$ 
should be smaller than $\ds \left| \pi_{E,L}(x) - C_{E,L} \frac{\sqrt{x}}{(\log x)^L} \right|$, just as in 
the case of the prime number theorem.
\end{Remark}

Consider Table 1, which lists the values of $\pi_{E,2}(x)$ for a few non-CM curves $E$ and 
various magnitudes $x$.
\begin{table}
\[
\begin{array}{|c||c|c|c|c|c|} \hline E & x=10^6 & x=10^8 & x=10^{10} & x = 10^{12} & x = 10^{13} \\ 
\hline\hline E_1: y^2 +y = x^3 - x & 0 & 1 & 16 & 115 & 332 \\ 
\hline E_2 : y^2 = x^3 + 6x - 2 & 0 & 5 & 32 & 208 & 563 \\ 
\hline E_3 : y^2 = x^3 - 3x + 4 & 0 & 0 & 0 & 0 & 0 \\ 
\hline
\end{array}
\]
\tablename{ 1:  Values of $\pi_{E,2}(x)$}
\end{table}
Note that $\pi_{E_2,2}(x)$ is larger than $\pi_{E_1,2}(x)$.  This difference 
is explained by the associated constants appearing in Conjecture \ref{mainconjecture}.  Indeed, a computation 
shows that
\[
\frac{C_{E_2,2}}{C_{E_1,2}} \approx 1.714.
\]
Also note that $\pi_{E_3,2}(10^{13}) = 0$.  The additional fact that 
$| \{ p \leq 10^6 : |E_3(\mbf_p)| \text{ is prime} \} | = 3236$ indicates that there probably are infinitely 
many primes $p$ for which $| E_3(\mbf_p) |$ is prime, in which case the above data suggests that $E_3$ might 
be a counterexample to Conjecture \ref{silvermanstangeconjecture}.  We will later see that 
$C_{E_3,2} = 0$, and that $E_3$ is indeed a counterexample, assuming a conjecture of 
Koblitz on the primality of $|E(\mbf_p)|$.

\begin{remark}
The heuristics which lead to Conjecture \ref{mainconjecture} are in the style of Koblitz and Lang-Trotter, 
whose conjectures have been proven ``on average over elliptic curves $E$'' (see \cite{bacoda} and \cite{dapa}).  
It might be interesting to see if one could also prove an average version of Conjecture \ref{mainconjecture}.
\end{remark}

\subsection{Positivity of $C_{E,L}$ and a directed graph $\mc{G}_E$}
In the interest of characterizing the non-CM elliptic curves which have infinitely many aliquot cycles of length 
$L$, we will state a graph-theoretic criterion for positivity of $C_{E,L}$.  Recall that a \textbf{directed graph} 
$\mc{G}$ is a pair $(\mc{V},\mc{E})$, where $\mc{V} = \mc{V}(\mc{G})$ is an arbitrary set of \textbf{vertices} and 
$\mc{E} = \mc{E}(\mc{G}) \subseteq \mc{V} \times \mc{V}$ is a subset of \textbf{directed edges}.   The 
sequence of vertices $(v_1, v_2, v_3, \dots, v_n)$ is a \textbf{closed walk of length $n$} if and only if 
$(v_i, v_{i+1}) \in \mc{E}$, for each $i \in \mbz/n\mbz = \{1, 2, 3, \dots, n \}$.  Note that closed walks may 
have repeated vertices.  For instance, if $(v,v) \in \mc{E}$ for some vertex $v$ (i.e. if $\mc{G}$ has a 
\emph{loop} at a vertex $v$), then $\mc{G}$ has closed walks of any length.

We will associate to an elliptic curve $E$ a directed graph $\mc{G}_E$.   First, consider the $n$-th division field $\mbq(E[n])$ of $E$, obtained by adjoining to $\mbq$ the $x$ and $y$-coordinates of the $n$-torsion $E[n]$ of a given Weierstrass model of $E$.  The extension $\mbq(E[n])$ is Galois over $\mbq$, and once we fix a basis over $\mbz/n\mbz$ of $E[n]$, we may view
\begin{equation} \label{subsetof}
\gal(\mbq(E[n])/\mbq) \subseteq GL_2(\mbz/n\mbz).
\end{equation}
We will now attach to $\gal(\mbq(E[n])/\mbq)$ a directed graph $\mc{G}_E(n)$.   Viewing Galois automorphisms as $2\times 2$ matrices via \eqref{subsetof}, the vertex set $\mc{V}(n)$ of our graph $\mc{G}_E(n)$ is
\[
\mc{V}(n) := \{ (t,d) \in \mbz/n\mbz \times (\mbz/n\mbz)^\times : \exists g \in \gal(\mbq(E[n])/\mbq) \text{ with } \tr g = t, \, \det g = d \}.    
\]
We define the set $\mc{E}(n) \subseteq \mc{V}(n) \times \mc{V}(n)$ of directed edges by declaring that 
$(v_1,v_2) \in \mc{E}(n)$ if and only if $d_1 + 1 - t_1 = d_2$, where $v_i = (t_i,d_i) \in \mc{V}(n)$.

Let $m_E$ denote the \textbf{torsion conductor} of $E$, which is defined as the smallest positive integer $m$ for which
\[
\forall n \in \mbz_{>0}, \quad \gal(\mbq(E[n])/\mbq) = \pi^{-1}(\gal(\mbq(E[\gcd(m,n)])/\mbq)),
\]
where $\pi : GL_2(\mbz/n\mbz) \rightarrow GL_2(\mbz/\gcd(m,n)\mbz)$ is the canonical projection.  (The existence of a torsion conductor $m_E$ for a non-CM elliptic curve $E$ is a celebrated theorem of Serre \cite{serre2}.)  Finally, we define the directed graph $\mc{G}_E$ to be the above graph at level $m_E$:
\[
\mc{G}_E := \mc{G}_E(m_E).
\]

The following version of Conjecture \ref{mainconjecture} states a criterion for positivity of $C_{E,L}$ in terms of the directed graph $\mc{G}_E$.
\begin{Conjecture} \label{mainconjecturewithcriterion}
Let $E$ be an elliptic curve over $\mbq$ without complex multiplication and $L \geq 2$ a fixed integer.  Suppose that the directed graph $\mc{G}_E$ has a closed walk of length $L$.  Then there are infinitely many aliquot cycles of length $L$ for $E$.  More precisely, there is a positive constant $C_{E,L} > 0$ so that, as $x \longrightarrow \infty$,
\[
\pi_{E,L}(x) \sim 
C_{E,L} \int_2^x \frac{1}{2\sqrt{t} (\log t)^L} dt.
\]
\end{Conjecture}

\begin{remark}
If $\mc{G}_E$ does not have a closed walk of length $L$, then $C_{E,L} = 0$ and there are at most finitely many 
aliquot cycles of length $L$ for $E$ (see Proposition \ref{equivproposition} below).
\end{remark}

In Section \ref{theconstant}, we will write down the constant $C_{E,L}$ explicitly as an ``almost Euler product'' 
and discuss its positivity in terms of the graph $\mc{G}_E$.  In Section \ref{heuristics}, we will develop the 
heuristics which lead to Conjecture \ref{mainconjecture}.  In Section \ref{numerics}, we will provide some 
numerical evidence for Conjecture \ref{mainconjecture} by examining the order of magnitude of 
$\ds \pi_{E,L}(x) - C_{E,L} \int_2^x \frac{1}{2\sqrt{t} (\log t)^L} \, dt$ for various elliptic curves $E$ and 
$L \in \{2, 3\}$.

\section{The constant} \label{theconstant}

We now describe in detail the constant $C_{E,L}$.  The following lemma allows us to interpret \eqref{aliquotdef} 
in terms of the Frobenius automorphisms\footnote{The Frobenius automorphism in $\gal(\mbq(E[n])/\mbq)$ attached to 
an unramified rational prime $p$ is only defined up to conjugation in $\gal(\mbq(E[n])/\mbq)$.  Here and 
throughout the paper, we understand $\frob_{\mbq(E[n])}(p)$ to be any choice of such a Frobenius automorphism.} 
$\frob_{\mbq(E[n])}(p_i) \in \gal(\mbq(E[n])/\mbq)$ attached to the various primes $p_i$.  Recall the trace of 
Frobenius $a_p(E) \in \mbz$, which satisfies the equation
\[
| E(\mbf_p) | =: p + 1 - a_p(E)
\]
as well as the Hasse bound
\begin{equation} \label{hassebound}
| a_p(E) | \leq 2\sqrt{p}.
\end{equation}
\begin{lemma} \label{wellknownlemma}
For any positive integer $n$ and any prime $p$ of good reduction for $E$ which does not divide $n$, $p$ is unramified in $\mbq(E[n])$ and for any Frobenius automorphism $\frob_{\mbq(E[n])}(p) \in \gal(\mbq(E[n])/\mbq)$, we have
\begin{equation*} 
\tr ( \frob_{\mbq(E[n])}(p)) \equiv a_p(E) \mod n
\end{equation*}
and
\begin{equation*} 
\det (\frob_{\mbq(E[n])}(p)) \equiv p \mod n.
\end{equation*}
\end{lemma}
\begin{proof}
See \cite[IV-4--IV-5]{serre}.
\end{proof}
For any subset $G \subseteq GL_2(\mbz/n\mbz)$, define 
\begin{equation*} 
G_{\text{ali-cycle}}^{L} := \left\{ (g_1,g_2, \dots, g_L) \in G^L : \; \forall i \in \mbz/L\mbz, \, 
\det(g_{i+1}) = \det(g_i) + 1 - \tr(g_i) \right\}.
\end{equation*}
Note that, by Lemma \ref{wellknownlemma}, if $(p_1,p_2, \dots, p_L)$ is an aliquot cycle of length $L$ for $E$, then
\begin{equation} \label{frobeniusinset}
(\frob_{\mbq(E[n])}(p_1), \frob_{\mbq(E[n])}(p_2), \dots, \frob_{\mbq(E[n])}(p_L)) \in \gal(\mbq(E[n])/\mbq)_{\text{ali-cycle}}^L.
\end{equation}
Next, let $\ds \phi(x) := \frac{2}{\pi} \sqrt{1 - x^2}$ be the distribution function of Sato-Tate, which 
(assuming $E$ has no CM) conjecturally\footnote{Assuming $E$ has non-integral $j$-invariant, the 
Sato-Tate conjecture is now a theorem of L. Clozel, M. Harris, N. Shepherd-Barron, and R. Taylor (see 
\cite{taylor} and the references therein).} satisfies
\begin{equation*} 
\lim_{x \rightarrow \infty} \frac{| \{ p \leq x : \frac{a_p(E)}{2\sqrt{p}} \in I \subseteq [-1,1] \} | }{| \{ p \leq x \} |} = \int_I \phi(x) dx.
\end{equation*}
In other words, $\phi$ is the density function of $a_p(E)/2\sqrt{p}$, viewed as a random variable.  Denote by 
$\phi_L := \phi * \phi * \dots * \phi$ the $L$-fold convolution of $\phi$ with itself, which (again assuming the
Sato-Tate conjecture) is the density function of the random variable
\[
\sum_{i = 1}^L \frac{a_{p_i}(E)}{2\sqrt{p_i}},
\]
provided the various terms $a_{p_i}(E)/2\sqrt{p_i}$ are ``statistically independent.''  Since the primes
$p_1, p_2, \dots, p_L$ belonging to an aliquot cycle must be close to one another (i.e. within $\approx L\sqrt{t}$ 
of one another where $p_1 \approx t$, by the Hasse bound \eqref{hassebound}), we are really assuming statistical 
independence \emph{in short intervals} of the various terms $a_{p_i}(E)/2\sqrt{p_i}$.
Finally, for a positive integer $k$, put
\begin{equation*} 
n_k := \prod_{p \leq k} p^k.
\end{equation*}
In Section \ref{heuristics}, we will develop heuristics which predict Conjecture \ref{mainconjecture}, with
\begin{equation} \label{Casinverselimit}
C_{E,L} := \frac{\phi_L(0)}{L} \cdot \lim_{k \rightarrow \infty} \frac{n_k^L | \gal(\mbq(E[n_k])/\mbq)_{\text{ali-cycle}}^L |}{| \gal(\mbq(E[n_k])/\mbq)^L |}.
\end{equation}

\subsection{The constant as a product}

We will presently prove the following proposition, which gives a more explicit expression of $C_{E,L}$ as a convergent Euler product.  Recall that $m_E$ denotes the torsion conductor of $E$, i.e. the smallest positive integer $m$ for which
\[
\forall n \in \mbz_{>0}, \quad \gal(\mbq(E[n])/\mbq) = \pi^{-1}(\gal(\mbq(E[\gcd(m,n)])/\mbq)),
\]
where $\pi : GL_2(\mbz/n\mbz) \rightarrow GL_2(\mbz/\gcd(m,n)\mbz)$ is the canonical projection.  
\begin{proposition} \label{explicitCprop}
For a positive integer $k$, let $\ds n_k := \prod_{p \leq k} p^k$.  Then one has
\[
\lim_{k \rightarrow \infty} \frac{n_k^L | \gal(\mbq(E[n_k])/\mbq)_{\text{ali-cycle}}^L |}{| \gal(\mbq(E[n_k])/\mbq)^L |} = \frac{m_E^L | \gal(\mbq(E[m_E])/\mbq)_{\text{ali-cycle}}^L |}{| \gal(\mbq(E[m_E])/\mbq)^L |} \cdot \prod_{\ell \nmid m_E} \frac{\ell^L | GL_2(\mbf_\ell)_{\text{ali-cycle}}^L |}{| GL_2(\mbf_\ell)^L |}
\]
Furthermore,
\begin{equation} \label{lefttoreader}
0 < \frac{\ell^L | GL_2(\mbf_\ell)_{\text{ali-cycle}}^L |}{| GL_2(\mbf_\ell)^L |} = 1 + O_L\left( \frac{1}{\ell^2} \right),
\end{equation}
so the infinite product $\ds \prod_{\ell \nmid m_E} \frac{\ell^L | GL_2(\mbf_\ell)_{\text{ali-cycle}}^L |}{| GL_2(\mbf_\ell)^L |}$ converges absolutely.
\end{proposition}
The proof of Proposition \ref{explicitCprop} involves the following two lemmas.
\begin{lemma} \label{splitslemma}
Let $n_1$ and $n_2$ be relatively prime positive integers, and pick any subgroups $G_1 \subseteq GL_2(\mbz/n_1\mbz)$ and $G_2 \subseteq GL_2(\mbz/n_2\mbz)$.  Then, viewing $G_1 \times G_2 \subseteq GL_2(\mbz/n_1n_2\mbz)$,
one has
\[
(G_1 \times G_2)_{\text{ali-cycle}}^L = (G_1)_{\text{ali-cycle}}^L \times (G_2)_{\text{ali-cycle}}^L.
\]
\end{lemma}

\noindent \emph{Proof of Lemma \ref{splitslemma}.}
Let $\iota : GL_2(\mbz/n_1\mbz) \times GL_2(\mbz/n_2\mbz) \rightarrow GL_2(\mbz/n_1n_2\mbz)$ be the isomorphism of the chinese remainder theorem, and set $G := \iota(G_1 \times G_2)$.  For each $L$-tuple $(g_i)_i \in G^L$, we have
\[
\forall i \in \mbz/L\mbz \; \det g_{i+1} \equiv \det g_i + 1 - \tr g_i \pmod{n_1n_2} \quad \Longleftrightarrow \quad 
\forall i \in \mbz/L\mbz \; \begin{array}{l} \det g_{i+1} \equiv \det g_i + 1 - \tr g_i \pmod{n_1} \\  
\det g_{i+1} \equiv \det g_i + 1 - \tr g_i \pmod{n_2}. \end{array}
\]
This implies the conclusion of Lemma \ref{splitslemma}. \hfill $\Box$
\begin{lemma} \label{stabilizeslemma}
Let $n$ be a positive integer and $n'$ any multiple of $n$ such that, for every prime numbur $\ell$, $\ell \mid n' \Rightarrow \ell \mid n$.  Let $\pi : GL_2(\mbz/n'\mbz) \rightarrow GL_2(\mbz/n\mbz)$ denote the canonical projection and let $G \subseteq GL_2(\mbz/n\mbz)$ be any subgroup.  Then one has
\begin{equation} \label{stabilityequality}
\frac{(n')^L | (\pi^{-1}(G))_{\text{ali-cycle}}^L |}{| \pi^{-1}(G)^L |} = \frac{n^L | G_{\text{ali-cycle}}^L|}{| G^L |}.
\end{equation}
\end{lemma}

\noindent \emph{Proof of Lemma \ref{stabilizeslemma}.}
By induction, it suffices to check the case $n' = \ell n$, where $\ell$ is some prime dividing $n$.  In this 
case, since $|\pi^{-1}(G)| = \ell^4 | G |$, \eqref{stabilityequality} is equivalent to 
\begin{equation} \label{equivstabilityequality}
| (\pi^{-1}(G) )_{\text{ali-cycle}}^L | = \ell^{3L} | G_{\text{ali-cycle}}^L |,
\end{equation}
which we now show.  Fix an element $g = (g_1, g_2, \dots, g_L) \in G_{\text{ali-cycle}}^L$, and note that any 
element $g' \in \pi^{-1}(g)$ has the form
\[
g' = (g_1',g_2', \dots, g_L') = (\tilde{g}_1(I + nA_1), \tilde{g}_2(I + nA_2), \dots, \tilde{g}_L(I + nA_L)) 
\in \pi^{-1}(g),
\]
where for each $i$, $\tilde{g}_i$ is any fixed lift to $GL_2(\mbz/\ell n\mbz)$ of $g_i$, and $A_i \in 
M_{2\times2}(\mbf_\ell)$ is arbitrary.  We will presently determine the exact conditions on the $A_i$ which 
force $(g_1',g_2', \dots, g_L') \in (\pi^{-1}(G) )_{\text{ali-cycle}}^L$.  First note that, since 
$(g_1, g_2, \dots, g_L) \in G_{\text{ali-cycle}}^L$, we must have
\begin{equation} \label{ginotin01}
\forall i \in \mbz/L\mbz, \quad g_i \pmod{\ell} \notin \{ 0, I \},
\end{equation}
and furthermore, the quantity 
\[
\gamma_i := \frac{\det \tilde{g}_{i+1} - \det \tilde{g}_i - 1 + \tr \tilde{g}_i}{n} \in \mbf_\ell
\]
is well-defined.  One checks that
\begin{equation} \label{gammaconversion}
\det g_{i+1}' \equiv \det g_i' + 1 - \tr g_i' \mod \ell n \; \Longleftrightarrow \; \gamma_i \equiv 
-\det g_{i+1} \cdot \tr A_{i+1} + \det g_i \cdot \tr A_i - \tr(g_i A_i) \mod \ell.
\end{equation}
The condition on the right-hand side is (affine) linear in the coefficients of $A_{i+1}$ and $A_i$.  We consider 
the linear transformation
\[
\begin{split}
&T : \mbf_\ell^{4L} \simeq M_{2\times 2}(\mbf_\ell)^L \rightarrow \mbf_\ell^L \\
&(A_i)_{i=1}^{L} \mapsto (-\det g_{i+1} \cdot \tr A_{i+1} + \det g_i \cdot \tr A_i - \tr(g_i A_i))_{i=1}^{L}.
\end{split}
\]
In light of \eqref{gammaconversion}, the condition \eqref{equivstabilityequality} will follow from the surjectivity
of the above linear transformation, which we now verify.  Writing coordinates as
\[
g_i =: \begin{pmatrix}
	x_i & y_i \\
	z_i & w_i
      \end{pmatrix}
\quad \text{ and } \quad
A_i =: \begin{pmatrix}
        a_i & b_i \\
	c_i & d_i
       \end{pmatrix},
\]
we have
\[
 T((A_i)) = ((\det g_i - x_i)a_i + (\det g_i - w_i)d_i - y_i c_i - z_i b_i - \det g_{i+1}a_{i+1} - \det g_{i+1}
d_{i+1}).
\]
By \eqref{ginotin01}, at least one of $\det g_i - x_i$, $\det g_i - w_i$, $y_i$ and $z_i$ must
be non-zero modulo $\ell$, and so
\[
 T( \{ 0 \} \times \dots \times \{ 0 \} \times M_{2\times 2}(\mbf_\ell) \times \{ 0 \} \times \dots
\times \{ 0 \} ) = \{ 0 \} \times \dots \times \{ 0 \} \times \mbf_\ell \times \{ 0 \} \times \dots
\times \{ 0 \},
\]
where the non-zero entries correspond to the same index $i$.  In particular, the linear transformation in question
is surjective and we have verified \eqref{equivstabilityequality}, finishing the proof of Lemma 
\ref{stabilizeslemma}.
\hfill $\Box$

\noindent \emph{Proof of Proposition \ref{explicitCprop}.}
Choose $k$ large enough so that $m_E \mid n_k$, and write $n_k = n_k^{(1)} \cdot n_k^{(2)}$, where $n_k^{(1)}$ is divisible by primes dividing $m_E$ and $\gcd(m_E,n_k^{(2)}) = 1$.   By definition of $m_E$, we then have
\[
\gal(\mbq(E[n_k])/\mbq) \simeq \pi^{-1}(\gal(\mbq(E[m_E])/\mbq)) \times \prod_{{\begin{substack} { \ell^k \parallel n_k \\ \ell \nmid m_E } \end{substack}}} GL_2(\mbz/\ell^k \mbz),
\]
where $\pi : GL_2(\mbz/n_k^{(1)}\mbz) \rightarrow GL_2(\mbz/m_E\mbz)$ is the canonical projection.
By Lemmas \ref{splitslemma} and \ref{stabilizeslemma}, we have
\[
\frac{n_k^L | \gal(\mbq(E[n_k])/\mbq)_{\text{ali-cycle}}^L |}{| \gal(\mbq(E[n_k])/\mbq)^L |} = \frac{m_E^L | \gal(\mbq(E[m_E])/\mbq)_{\text{ali-cycle}}^L |}{| \gal(\mbq(E[m_E])/\mbq)^L |} \cdot \prod_{{\begin{substack} { \ell \mid n_k \\ \ell \nmid m_E } \end{substack}}} \frac{\ell^L | \gal(\mbq(E[\ell])/\mbq)_{\text{ali-cycle}}^L |}{| \gal(\mbq(E[\ell])/\mbq)^L |}.
\]
Taking the limit as $k \rightarrow \infty$, we arrive at the product representation of $C_{E,L}$ stated in Proposition \ref{explicitCprop}.  We leave the verification of \eqref{lefttoreader} as an exercise.
\hfill $\Box$

\subsection{Positivity of the constant}

We will now discuss the positivity of $C_{E,L}$.  The following corollary of Proposition \ref{explicitCprop} is immediate.
\begin{cor} \label{equivalencecor}
One has
\begin{equation*} 
C_{E,L} > 0 \; \Longleftrightarrow \; \gal(\mbq(E[m_E])/\mbq)_{\text{ali-cycle}}^L \neq \emptyset.
\end{equation*}
\end{cor}

We will now prove the following proposition, which allows one to deduce Conjecture \ref{mainconjecturewithcriterion} from Conjecture \ref{mainconjecture}.
\begin{proposition} \label{equivproposition}
For any non-CM elliptic curve $E$ over $\mbq$, one has
\begin{equation} \label{positivitycriterion}
C_{E,L} > 0 \; \Longleftrightarrow \; \mc{G}_E \text{ has a closed walk of length $L$}.
\end{equation}
Furthermore, if $\mc{G}_E$ has no closed walks of length $L$, then there are only finitely many aliquot cycles $(p_1,p_2, \dots, p_L)$ of length $L$ for $E$.
\end{proposition}
\begin{proof}
First we prove \eqref{positivitycriterion}.  By Corollary \ref{equivalencecor}, we are reduced to showing that
\begin{equation} \label{finitepositivitycriterion}
\gal(\mbq(E[m_E])/\mbq)_{\text{ali-cycle}}^L \neq \emptyset \; \Longleftrightarrow \; \mc{G}_E \text{ has a closed walk of length $L$}.
\end{equation}
The mapping
\[
\begin{split}
\gal(\mbq(E[m_E])/\mbq) &\rightarrow \mc{V}(\mc{G}_E) \\
g &\mapsto (\tr g, \det g)
\end{split}
\]
induces a mapping 
$\gal(\mbq(E[m_E])/\mbq)_{\text{ali-cycle}}^L \longrightarrow \{ \text{closed walks of length $L$ in } 
\mc{G}_E \}$.  Thus, if \newline $\gal(\mbq(E[m_E])/\mbq)_{\text{ali-cycle}}^L \neq \emptyset$ then $\mc{G}_E$ has 
a closed walk of length $L$.  Conversely, suppose $\mc{G}_E$ has a closed walk $(v_1,v_2,v_3,\dots,v_L)$ of 
length $L$.  Recall that $\mc{V} = \mbz/m_E\mbz \times (\mbz/m_E\mbz)^\times$ and write $v_i = (t_i, d_i)$.  
Choosing any element $g_i \in \gal(\mbq(E[m_E])/\mbq)$ with $\tr g_i = t_i$ and $\det g_i = d_i$, we have then 
constructed an element $(g_1, g_2, \dots, g_L) \in \gal(\mbq(E[m_E])/\mbq)_{\text{ali-cycle}}^L$, so that 
$\gal(\mbq(E[m_E])/\mbq)_{\text{ali-cycle}}^L \neq \emptyset$.  By Corollary \ref{equivalencecor}, we conclude 
the proof of \eqref{positivitycriterion}.  

To see why the nonexistence of closed walks of length $L$ in $\mc{G}_E$ implies that 
$\ds \lim_{x \rightarrow \infty} \pi_{E,L}(x) < \infty$, note first that, by \eqref{finitepositivitycriterion}, one 
has that $\gal(\mbq(E[m_E])/\mbq)_{\text{ali-cycle}}^L = \emptyset$.  But then \eqref{frobeniusinset} implies 
that $\ds \lim_{x \rightarrow \infty} \pi_{E,L}(x) < \infty$, and the proof of 
Proposition \ref{equivproposition} is complete.
\end{proof}

\section{Heuristics} \label{heuristics}

We will construct a probabilistic model in the style of Koblitz \cite{koblitz} and Lang-Trotter \cite{langtrotter}.  
We shall call the $L$-tuple $(p_1, p_2, \dots, p_L)$ of distinct prime numbers an \textbf{aliquot sequence of length $L$ for $E$} if it satisfies
\[
p_{i+1} = | E(\mbf_{p_i}) | \quad \forall i \in \{1, 2, \dots L-1 \}.
\]
Thus, an aliquot cycle of length $L$ is an aliquot sequence of length $L$ which additionally satisfies $p_1 = | E(\mbf_{p_L})|$.  Suppose that $(p_1,p_2, \dots,p_L)$ is an aliquot sequence of length $L$ for $E$.  By substituting $p_2 = p_1 + 1 - a_{p_1}(E)$ into the equation $p_3 = p_2 + 1 - a_{p_2}(E)$, one finds that $p_3 = p_1 + 2 - (a_{p_1}(E) + a_{p_2}(E))$, and continuing in this manner one obtains
\begin{equation} \label{conditionontraces}
p_1 = |E(\mbf_{p_L})|  \; \Longleftrightarrow \; \sum_{j = 1}^L a_{p_j}(E) = L.
\end{equation}
Thus, a given $L$-tuple $(p_1,p_2, \dots, p_L)$ of positive integers is an aliquot cycle of length $L$ for $E$ if and only if the following conditions hold:
\begin{enumerate}
 \item[($1_L$)] The $L$-tuple $(p_1,p_2, \dots, p_L)$ is an aliquot sequence of length $L$ for $E$.
 \item[($2_L$)] One has $\displaystyle \sum_{j = 1}^L a_{p_j}(E) = L$.
\end{enumerate}
Consider the following condition, which generalizes condition ($2_L$) above by replacing $L$ with an arbitrary fixed integer $r$:
\begin{enumerate}
 \item[($2'_L$)] One has $\displaystyle \sum_{j = 1}^L a_{p_j}(E) = r$.
\end{enumerate}

We will now develop the heuristic ``probability'' that a given $L$-tuple $(p_1,p_2, \dots, p_L)$ of positive 
integers satisfies ($1_L$) and ($2'_L$).  First, we must gather some notation.  Fix a positive integer 
$n$ and elements $a,b \in \mbz/n\mbz$.  For any subset $S \subseteq GL_2(\mbz/n\mbz)$, let 
\[
\begin{split}
S_{\mc{N} = a} &:= \{ g \in S : \det(g) + 1 - \tr(g) = a \} \\
S^{\det = b} &:= \{ g \in S : \det(g) = b \} \\
S_{\mc{N} = a}^{\det = b} &:= S_{\mc{N} = a} \cap S^{\det = b}.
\end{split}
\]
Finally, for $L \geq 1$ and $G \subseteq GL_2(\mbz/n\mbz)$, put
\[
G_{\text{ali-sequence}}^L := \left\{ (g_1,g_2, \dots, g_L) \in G^L : \; \forall i \in \{ 1, 2, \dots, L-1 \}, \, \det(g_{i+1}) = \det(g_i) + 1 - \tr(g_i) \right\}.
\]
Note that if $L = 1$, the defining conditions become empty and we have $G_{\text{ali-seqence}}^{L=1} = G$.  For a general $L \geq 1$, note that any aliquot sequence $(p_1, p_2, \dots, p_L)$ for $E$ will satisfy
\[
(\frob_{\mbq(E[n])}(p_1), \frob_{\mbq(E[n])}(p_2), \dots \frob_{\mbq(E[n])}(p_L)) \in \gal(\mbq(E[n])/\mbq)_{\text{ali-sequence}}^L.
\]
Finally, for a fixed integer $r$, define
\[
G_{\text{ali-sequence}}^{L,\, \sum \tr = r} :=\left\{ (g_1, g_2, \dots, g_L) \in G_{\text{ali-sequence}}^L : \sum_{i = 1}^L \tr(g_i) \equiv r \mod n \right\}.
\]

We will presently derive an expression for the probability 
\[
\mc{P}_{(1_L), (2'_L)}(t) := \text{Prob}\left( \text{$(p_1,p_2,\dots, p_L)$ satisfies ($1_L$) and ($2'_L$), given that $p_1 \approx t$} \right),
\]
Putting $\mc{P}_{(1_L)}(t)$ for the probability that $(p_1,p_2,\dots,p_L)$ satisfies ($1_L$) above, and $\mc{P}_{(2'_L)}^{\,\text{given }(1_L)}(t)$ for the conditional probability that $(p_1,p_2,\dots, p_L)$ satisfies ($2'_L$), given that it satisfies ($1_L$), we have
\begin{equation} \label{productofprobs}
\mc{P}_{(1_L),(2'_L)}(t) = \mc{P}_{(1_L)}(t) \cdot \mc{P}_{(2'_L)}^{\,\text{given }(1_L)}(t).
\end{equation}
In Section \ref{onetwoandthree} below, we will derive the probability formula
\begin{equation} \label{probonetwothree}
\mc{P}_{(1_L)}(t) \approx \frac{n^{L-1} \cdot | \gal(\mbq(E[n])/\mbq)_{\text{ali-sequence}}^L |}{| \gal(\mbq(E[n])/\mbq)^L |} \cdot \frac{1}{(\log t)^L}.
\end{equation}
Following this, in Section \ref{fourprime}, we will derive
\begin{equation} \label{probfourprime}
\mc{P}_{(2'_L)}^{\,\text{given }(1_L)}(t) \approx \phi_L\left( \frac{r}{2\sqrt{t}} \right) \frac{n \cdot | \gal(\mbq(E[n])/\mbq)_{\text{ali-sequence}}^{L, \, \sum \tr = r} |}{| \gal(\mbq(E[n])/\mbq)_{\text{ali-sequence}}^L |} \cdot \frac{1}{2\sqrt{t}}.
\end{equation}

Before deriving \eqref{probonetwothree} and \eqref{probfourprime}, we will now observe that, taken together, they lead to Conjecture \ref{mainconjecture}.  Indeed, using \eqref{productofprobs}, \eqref{probonetwothree} and \eqref{probfourprime},
one concludes
\[
\mc{P}_{(1_L),(2'_L)}(t) \approx \phi_L\left( \frac{r}{2\sqrt{t}} \right) \cdot \frac{n^{L} | \gal(\mbq(E[n])/\mbq)_{\text{ali-sequence}}^{L, \sum \tr = r} |}{| \gal(\mbq(E[n])/\mbq)^L |} \cdot \frac{1}{2\sqrt{t} (\log t)^L}
\]
Just as with \eqref{conditionontraces}, one verifies that, for each $(g_1, g_2, \dots, g_L) \in GL_2(\mbz/n\mbz)_{\text{ali-sequence}}^L$, one has
\[
\det(g_L) + 1 - \tr(g_L) = \det g_1 \; \Longleftrightarrow \; \sum_{i = 1}^L \tr(g_i) \equiv L \mod n.
\]
It follows that $\gal(\mbq(E[n])/\mbq)_{\text{ali-cycle}}^L = \gal(\mbq(E[n])/\mbq)_{\text{ali-sequence}}^{L, \, \sum \tr = L}$.  Thus, putting $r = L$, $n = n_k$ and taking the limit as $k \rightarrow \infty$, one arrives at
\[
\mc{P}_{(1_L),(2_L)}(t) \approx \phi\left( \frac{L}{2\sqrt{t}} \right) \cdot \lim_{k \rightarrow \infty} \frac{n_k^L | \gal(\mbq(E[n_k])/\mbq)_{\text{ali-cycle}}^L |}{| \gal(\mbq(E[n_k])/\mbq)^L |} \cdot \frac{1}{2\sqrt{t} (\log t)^L}.
\]
Thus, using
\[
\pi_{E,L}(x) \approx \frac{1}{L} \int_2^x \mc{P}_{(1_L),(2_L)}(t) \, dt,
\]
one arrives at Conjecture \ref{mainconjecture}.  The reason for the extra factor of $L$ in the denominator above is that $\pi_{E,L}(x)$ counts \emph{normalized} aliquot cycles, whereas the heuristic probabilities above do not take normalization into account.  Also, since $L$ is fixed, one verifies that the estimation $\phi(L/(2\sqrt{t})) \approx \phi(0)$ does not affect the asymptotic.

\subsection{The probability that $(p_1,p_2, \dots, p_L)$ satisfies ($1_L$)} \label{onetwoandthree}

We will now derive a refined probability formula which implies \eqref{probonetwothree}.  Fix a vector
${\bf{a}} = (a_2, a_3, \dots, a_L) \in ((\mbz/n\mbz)^\times)^{L-1}$,
and consider the probability
\[
\mc{P}_{(1_L)}^{{\bf{a}}}(t) := \text{Prob}( \text{$(p_1,p_2, \dots p_L)$ satisfies ($1_L$) and $\forall i \in \{ 2, 3, \dots, L \}, \, p_i \equiv a_i \mod n$} )
\]
and (for any subset $G \subseteq GL_2(\mbz/n\mbz)$) the subset
\[
G_{\text{ali-sequence}}^{L, \, {\bf{a}}} := \{ (g_1, g_2, \dots, g_L) \in G_{\text{ali-sequence}}^{L} : \forall i \in \{ 2, 3, \dots, L \}, \, \det(g_i) = a_i \}.
\]
In case $L = 1$, the vector ${\bf{a}} \in ((\mbz/n\mbz)^\times)^0$ is non-existent, and as before we interpret the empty condition as 
$G_{\text{ali-sequence}}^{1, \, {\bf{a}}} = G$.  
Also note the decomposition
\begin{equation} \label{Omegadecomposition}
G_{\text{ali-sequence}}^{L, \, {\bf{a}}} = G_{\mc{N} = a_2} \times G_{\mc{N} = a_3}^{\det = a_2} \times G_{\mc{N} = a_4}^{\det = a_3} \times \dots \times G_{\mc{N} = a_L}^{\det = a_{L-1}} \times G^{\det = a_L}.
\end{equation}
Finally, note that if ${\bf{a_1}} \neq {\bf{a_2}}$, then $G_{\text{ali-sequence}}^{L, \, {\bf{a_1}}} \cap G_{\text{ali-sequence}}^{L, \, {\bf{a_2}}} = \emptyset$, and so we have a disjoint union
\[
G_{\text{ali-sequence}}^{L} = \bigsqcup_{{\begin{substack} { {\bf{a}} \in ((\mbz/n\mbz)^\times)^{L-1}} \end{substack}}} G_{\text{ali-sequence}}^{L, \, {\bf{a}}}.
\]
For similar reasons, we have
\[
\mc{P}_{(1_L)}(t) = \sum_{{\begin{substack} { {\bf{a}} \in ((\mbz/n\mbz)^\times)^{L-1}} \end{substack}}} \mc{P}_{(1_L)}^{{\bf{a}}}(t).
\]
Thus, \eqref{probonetwothree} will follow from
\begin{equation} \label{probonetwothreea}
\mc{P}_{(1_L)}^{{\bf{a}}}(t) \approx \frac{n^{L-1} \cdot | \gal(\mbq(E[n])/\mbq)_{\text{ali-sequence}}^{L, \, {\bf{a}}} |}{| \gal(\mbq(E[n])/\mbq)^L |} \cdot \frac{1}{(\log t)^L},
\end{equation}
which we will now derive by induction on $L$. \\

\noindent \emph{Base case:  $L = 1$}.  Suppose that $p_1$ is a positive integer of size about $t$.  One may interpret the prime number theorem as the probabilistic statement that
\[
\mc{P}_{(1_{L=1})}(t) = \text{Prob}(\text{$p_1$ is prime}) \approx \frac{1}{\log t},
\]
which is base case $L = 1$ of \eqref{probonetwothreea}.  \\

\noindent \emph{Induction step}.  Assume now that \eqref{probonetwothreea} holds for some fixed $L \geq 1$, and fix any vector ${\bf{a}} = (a_2, a_3, \dots, a_{L+1}) \in ((\mbz/n\mbz)^\times)^L$.  Since the statement
\[
\text{``$(p_1,p_2, \dots p_{L+1})$ satisfies ($1_{L+1}$) and $\forall i \in \{ 2, 3, \dots, L+1 \}, \, p_i \equiv a_i \mod n$''}
\]
is equivalent to
\[
\text{$(p_1,p_2, \dots p_L)$ satisfies ($1_L$) and $\forall i \in \{ 2, 3, \dots, L \}, \, p_i \equiv a_i \mod n$}
\]
\[
\text{ and }
\]
\[
\text{$p_{L+1} := p_L + 1 - a_{p_L}(E)$ is prime, and $p_{L+1} \equiv a_{L+1} \mod n$,}
\]
we see that 
\begin{equation} \label{Poft}
\mc{P}_{(1_{L+1})}^{(a_2, a_3, \dots, a_L, a_{L+1})}(t) = \mc{P}_{(1_L)}^{(a_2, a_3, \dots, a_L)}(t) \cdot \mc{P}(t),
\end{equation}
where $\mc{P}(t)$ is the conditional probability that $p_{L+1} := p_L + 1 - a_{p_L}(E)$ is prime, and that 
$p_{L+1} \equiv a_{L+1} \mod n$, given that ($1_L$) holds.  To estimate $\mc{P}(t)$, let us assume that ($1_L$) 
holds.  First note that, by the Hasse bound $|a_p(E)| \leq 2\sqrt{p}$, one has
\[
p_{L+1} = p_1 + L - \sum_{i = 1}^L a_{p_i}(E) \in [ p_1 + L - 2L\sqrt{p_{\text{max}}}, p_1 + L + 2L\sqrt{p_{\text{max}}}],
\]
where $p_{\text{max}} := \max \{ p_i : i = 1, 2, \dots, L \}$.  By induction we have $p_{\text{max}} = t + O_L(\sqrt{t})$, and so $p_{L+1} \approx t$, with an error of $O_L(\sqrt{t})$.  
Now, if $p_{L+1}$ were a positive integer of size about $t$ selected independently of $(p_1,p_2, \dots, p_L)$, then 
\begin{equation} \label{wrong}
\text{Prob}(\text{$p_{L+1}$ is prime and $p_{L+1} \equiv a_{L+1} \mod n$}) \approx \frac{1}{\varphi(n) \log t},
\end{equation}
by the prime number theorem in arithmetic progressions.  If the positive integer $p_{L+1}$ were chosen randomly and independently of the previous primes, then the probability that $p_{L+1} \equiv a_{L+1} \mod n$ would be $1/n$.  However, $p_{L+1}$ is not chosen independently of $(p_1,p_2, \dots, p_L)$; it is related to $p_L$ by the formula $p_{L+1} = p_L + 1 - a_{p_L}(E)$.  Thus, the congruence $p_{L+1} \equiv a_{L+1} \mod n$ is really the demand that $\frob_{\mbq(E[n])}(p_L) \in \gal(\mbq(E[n])/\mbq)_{\mc{N} = a_{L+1}}$.  Since we assume that ($1_L$) holds, we know that $\frob_{\mbq(E[n])}(p_L) \in GL_2(\mbz/n\mbz)^{\det = a_{L}}$.   It is thus natural to multiply \eqref{wrong} by the correction factor
\[
\frac{|  \gal(\mbq(E[n])\mbq)_{\mc{N}= a_{L+1}}^{\det = a_L} | / | \gal(\mbq(E[n])/\mbq)^{\det = a_L} | }{1/n}, 
\]
obtaining
\begin{equation} \label{right}
\begin{split}
\mc{P}(t) &\approx 
\frac{|\gal(\mbq(E[n])\mbq)_{\mc{N}= a_{L+1}}^{\det = a_L}|/|\gal(\mbq(E[n])/\mbq)^{\det = a_L}|}{1/n} \cdot 
\frac{1}{\varphi(n) \log t} \\
&= \frac{n | \gal(\mbq(E[n])/\mbq)_{\mc{N} = a_{L+1}}^{\det = a_L} |}{ | \gal(\mbq(E[n])/\mbq) |} \cdot 
\frac{1}{\log t}.
\end{split}
\end{equation}
By \eqref{Omegadecomposition}, we may re-write \eqref{probonetwothreea} as
\[
\mc{P}_{(1_L)}^{{\bf{a}}}(t) \approx n^{L-1} \cdot \frac{| \gal(\mbq(E[n])/\mbq)_{\mc{N} = a_2} |}{| \gal(\mbq(E[n])/\mbq) |} \cdot \left( \prod_{i = 2}^{L-1} \frac{| \gal(\mbq(E[n])/\mbq)_{\mc{N} = a_{i+1}}^{\det = a_i} |}{| \gal(\mbq(E[n])/\mbq) |} \right) \cdot \frac{| \gal(\mbq(E[n])/\mbq)^{\det = a_L} |}{| \gal(\mbq(E[n])/\mbq) |} \cdot \frac{1}{(\log t)^L}.
\]
Plugging this expression and \eqref{right} into \eqref{Poft}, and using the fact that 
\[
|\gal(\mbq(E[n])/\mbq)^{\det = a_L} | = |\gal(\mbq(E[n])/\mbq)^{\det = a_{L+1}} |, 
\]
one concludes the induction step, completing the derivation of \eqref{probonetwothreea}, and thus of \eqref{probonetwothree}.  

As a byproduct of our analysis, we have motivated the following conjecture, wherein
\[
\pi_{E,L}^{\text{ali-sequence}}(x) := | \{ p_1 \leq x : \text{ $\exists$ an aliquot sequence $(p_1,p_2, \dots, p_L)$ for $E$} \} |
\]
and
\[
C_{E,L}^{\text{ali-sequence}} := \lim_{k \rightarrow \infty} \frac{n_k^{L-1} \cdot | \gal(\mbq(E[n_k])/\mbq)_{\text{ali-sequence}}^L |}{| \gal(\mbq(E[n_k])/\mbq)^L |}.
\]
\begin{Conjecture} \label{alisequenceconjecture}
Let $E$ be an elliptic curve over $\mbq$ without complex multiplication and $L \geq 2$ a fixed integer.  Then as $x \longrightarrow \infty$, one has
\begin{equation*} 
\pi_{E,L}^{\text{ali-sequence}}(x) \sim 
C_{E,L}^{\text{ali-sequence}} \int_2^x \frac{1}{(\log t)^L} dt.
\end{equation*}
\end{Conjecture}
Similarly to Proposition \ref{equivproposition}, one has
\[
C_{E,L}^{\text{ali-sequence}} > 0 \; \Longleftrightarrow \; \mc{G}_E \text{ has a (directed) walk of length $L$.}
\]

\subsection{The conditional probability that $(p_1,p_2, \dots, p_L)$ satisfies ($2'_L$)} \label{fourprime}

We will now derive \eqref{probfourprime}, completing the heuristic derivation of Conjecture \ref{mainconjecture}.  Suppose that $(p_1, p_2, \dots, p_L)$ is an aliquot sequence of length $L$ for $E$, i.e. that it satisfies ($1_L$).  What is the conditional probability that $\ds \sum_{i = 1}^L a_{p_i}(E) = r$?  In the case $L = 1$, condition ($1_L$) is empty, and our question becomes identical to the Lang-Trotter conjecture for fixed Frobenius trace. In what follows, we will develop a probabilistic model in the same style as theirs.

Fixing a level $n$, the number $f_n(r,p) \geq 0$ will estimate the probability of the event that $\ds \sum_{i = 1}^L a_{p_i}(E) = r$, given that $(p = p_1, p_2, \dots, p_L)$ is an aliquot sequence of length $L$ for $E$.  We will model the situation by assuming that the vector
\begin{equation} \label{frobvector}
(\frob_{\mbq(E[n])}(p_1), \frob_{\mbq(E[n])}(p_2), \dots \frob_{\mbq(E[n])}(p_L)) \in \gal(\mbq(E[n])/\mbq)_{\text{ali-sequence}}^L
\end{equation}
is randomly distributed according to counting measure, and we will assume that the various $\ds \frac{a_{p_i}(E)}{2\sqrt{p_i}}$ are independent at infinity, i.e. that $\phi_L$ is the distribution function for $\ds \sum_{i = 1}^L \frac{a_{p_i}(E)}{2\sqrt{p_i}}$.  We will also assume independence of the random variables $\ds \sum_{i = 1}^L \frac{a_{p_i}(E)}{2\sqrt{p_i}}$ and \eqref{frobvector}.  
Finally, in order to simplify our model, we will also regard all of the various primes $p_i$ as having the same 
size, namely $p$.  These considerations lead us to the following assumptions about the probabilities $f_n(r,p)$:
\begin{equation*} 
\begin{split}
f_n(r,p) &= 0 \; \text{ if } \; |r| > 2L \sqrt{p} \\
f_n(r,p) &= \phi_L\left( \frac{r}{2\sqrt{p}} \right) \cdot 
\frac{n | \gal(\mbq(E[n])/\mbq)_{\text{ali-sequence}}^{L, \sum \tr = r} |}{| \gal(\mbq(E[n])/\mbq)_{\text{ali-sequence}}^L |} 
\cdot c_p \; \text{ if } \; |r| \leq 2L \sqrt{p},
\end{split}
\end{equation*}
where $c_p$ is some constant chosen so that $\ds \sum_{r \in \mbz} f_n(r,p) = 1$.  Then, similarly to \cite[pp. 31--32]{langtrotter}, one concludes that $\ds c_p \sim \frac{1}{2\sqrt{p}}$, as $p \rightarrow \infty$.  This leads to \eqref{probfourprime}, completing the derivation of Conjecture \ref{mainconjecture}.

\section{Examples} \label{numerics}

We will now give some numerical evidence for Conjecture \ref{mainconjecture}. 

\subsection{Elliptic curves with $C_{E,L} > 0$}

Table 2 and Table 3 display some data for four elliptic curves.  In each table, the column labelled ``Predicted'' 
lists the approximate values of $\ds C_{E,L} \int_2^{10^{13}} \frac{dt}{2\sqrt{t}(\log t)^L}$, ``Actual'' lists the values of 
$\pi_{E,L}(10^{13})$, and ``$\%$ error'' lists as a percentage the approximate values of
\[
\frac{C_{E,L} \int_2^{10^{13}} \frac{dt}{2\sqrt{t}(\log t)^L} - 
\pi_{E,L}(10^{13})}{C_{E,L} \int_2^{10^{13}} \frac{dt}{2\sqrt{t}(\log t)^L}}.
\]
The first and third curves were already considered in \cite{silvermanstange}, and are included here largely 
to show the contrast with the second curve.  For each of these curves, a detailed list of the aliquot cycles with 
$p_1 \leq 10^{13}$ may be found in the appendix.

\begin{table}
\begin{equation*} 
\begin{array}{|c||c|c|c|} \hline E & \text{Predicted} & \text{Actual} & \text{\% error} \\ 
\hline\hline y^2 +y = x^3 - x & 318.98 & 332 & -4.08 \% \\ 
\hline y^2 = x^3 + 6x - 2 & 546.78 & 563 & -2.97 \% \\
\hline y^2 + y = x^3 + x^2 & 318.97 & 328 & -2.83 \% \\
\hline y^2 + xy + y = x^3 - x^2 & 318.95 & 331 & -3.78 \% \\
\hline
\end{array}
\end{equation*}
\tablename{ 2:  Data on $\pi_{E,2}(10^{13})$ for various $E$}
\end{table}

\begin{table}
\begin{equation*} 
\begin{array}{|c||c|c|c|} \hline E & \text{Predicted} & \text{Actual} & \text{\% error} \\ 
\hline\hline y^2 +y = x^3 - x & 3.03 & 3 & 1.05 \% \\ 
\hline y^2 = x^3 + 6x - 2 & 12.59 & 12 & 4.66 \% \\
\hline y^2 + y = x^3 + x^2 & 3.04 & 2 & -34.10 \% \\
\hline y^2 + xy + y = x^3 - x^2 & 3.02 & 4 & -32.48 \% \\
\hline
\end{array}
\end{equation*}
\tablename{ 3:  Data on $\pi_{E,3}(10^{13})$ for various $E$}
\end{table}

The four elliptic curves $E$ under consideration satisfy the property that, for each $n \geq 1$, 
\begin{equation} \label{serrecurve}
[ GL_2(\mbz/n\mbz) : \gal(\mbq(E[n])/\mbq) ] \leq 2
\end{equation}
(See \cite[pp. 309--311]{serre} and \cite[p. 51]{langtrotter}).  As shown in \cite[pp. 310--311]{serre}, 
this is the smallest index that one can have for general $n$ when the elliptic curve $E$ is defined over $\mbq$. 
We call any elliptic curve $E$ satisfying \eqref{serrecurve} a \textbf{Serre curve}.  Serre curves are thus
elliptic curves for which $\gal(\mbq(E[n])/\mbq)$ is ``as large as possible for all $n$,'' and it has been shown
that, when ordered by height, almost all elliptic curves are Serre curves (see \cite{jones}, \cite{radhakrishnan}).
One can show that for any Serre curve $E$, one has $C_{E,L} > 0$.  In fact, if we define the constant 
$C_L$ by
\[
C_L := \frac{\phi_L(0)}{L} \cdot \lim_{k \rightarrow \infty} \frac{n_k^L | GL_2(\mbz/n_k\mbz)_{\text{ali-cycle}}^L |}{| GL_2(\mbz/n_k\mbz)^L |} = \frac{\phi_L(0)}{L} \cdot \prod_{\ell \text{ prime}} \frac{\ell^L | GL_2(\mbf_\ell)_{\text{ali-cycle}}^L |}{| GL_2(\mbf_\ell)^L |},
\]
then for any Serre curve $E$ one has that
\[
C_{E,L} = C_L \cdot f_L(\gD_{sf}(E)),
\]
where $\gD_{sf}(E)$ denotes the square-free part of the discriminant of any Weierstrass model of $E$ and $f_L$ is 
a positive function which approaches $1$ as $|\gD_{sf}(E)|$ approaches infinity.  For $L = 2$ one has
\[
\begin{split}
C_2 &= \frac{\phi_2(0)}{2} \cdot \prod_{\ell \text{ prime}} 
\frac{\ell^2 | GL_2(\mbf_\ell)_{\text{ali-cycle}}^2 |}{| GL_2(\mbf_\ell)^2 |} \\
&= \frac{8}{3\pi^2} \cdot \prod_{\ell \text{ prime}} 
\frac{\ell^2(\ell^4 - 2\ell^3 - 2\ell^2 + 3\ell + 3)}{[(\ell^2 - 1)(\ell - 1)]^2} \\
&\approx 0.077088124,
\end{split}
\]
whereas for $L = 3$ one has
\[
\begin{split}
C_3 &= \frac{\phi_3(0)}{3} \cdot \prod_{\ell \text{ prime}} 
\frac{\ell^3 | GL_2(\mbf_\ell)_{\text{ali-cycle}}^3 |}{| GL_2(\mbf_\ell)^3 |} \\
&= \frac{\phi_3(0)}{3} \cdot \prod_{\ell \text{ prime}} 
\frac{\ell^3[\ell^6 - 3\ell^5 - 3\ell^4 + 14\ell^3 + 
(3 + \chi(\ell))\ell^2 - (19 + 3\chi(\ell))\ell - 10 - 3\chi(\ell)]}{[(\ell^2 - 1)(\ell - 1)]^3} \\
&\approx 0.019759298,
\end{split}
\]
where $\ds \chi(\ell) = \left( \frac{-3}{\ell} \right)$ denotes the character of conductor $3$.
Table 4 gives the values of $C_{E,2}$, $C_{E,3}$ and $\gD_{sf}(E)$ for each of the four curves under consideration. 
\begin{table}
\begin{equation*}
\begin{array}{|c||c|c|c|} \hline E & C_{E,2} & C_{E,3} & \gD_{sf}(E) \\ 
\hline\hline y^2 +y = x^3 - x & \approx 0.077093 & \approx 0.019841 & 37 \\ 
\hline y^2 = x^3 + 6x - 2 & \approx 0.132151 & \approx 0.082365 & -3 \\
\hline y^2 + y = x^3 + x^2 & \approx 0.077091 & \approx 0.019861 & -43 \\
\hline y^2 + xy + y = x^3 - x^2 & \approx 0.077088 & \approx 0.019759 & -53 \\
\hline
\end{array}
\end{equation*}
\tablename{ 4:  Values of $C_{E,2}$, $C_{E,3}$ and $\gD_{sf}(E)$}
\end{table}
The reason the second curve has a larger value of $C_{E,L}$ is that $|\gD_{sf}(E)|$ is smaller 
for this curve than for the others.

\bigskip

\subsection{An elliptic curve with $C_{E,L} = 0$}

We will now discuss briefly the elliptic curve 
\begin{equation} \label{counterexampleE}
E \; : \; y^2 = x^3 -3x + 4 
\end{equation}
which was mentioned in the introduction, for which $\pi_{E,L}(x) \equiv 0$ and whose associated graph $\mc{G}_E$ 
contains no closed walks at all.  We will presently describe the Galois group $\gal(\mbq(E[4])/\mbq)$, which is an 
index $4$ subgroup of $GL_2(\mbz/4\mbz)$.  First, define the subgroup $H(4) \subseteq GL_2(\mbz/4\mbz)$ by 
\[
H(4) := \left\{ \begin{pmatrix} 1 & 0 \\ 0 & 1 \end{pmatrix}, \begin{pmatrix} 0 & 1 \\ -1 & -1 \end{pmatrix}, 
\begin{pmatrix} -1 & -1 \\ 1 & 0 \end{pmatrix}, \begin{pmatrix} -1 & -1 \\ 0 & 1 \end{pmatrix}, 
\begin{pmatrix} 1 & 0 \\ -1 & -1 \end{pmatrix}, \begin{pmatrix} 0 & 1 \\ 1 & 0 \end{pmatrix} \right\}.
\]
We then have
\begin{equation} \label{galoisatlevel4}
\gal(\mbq(E[4])/\mbq) = H(4) \cdot \left( I + 2 \left\{ \begin{pmatrix} 0 & 0 \\ 0 & 0 \end{pmatrix}, 
\begin{pmatrix} 1 & 1 \\ 0 & 1 \end{pmatrix}, \begin{pmatrix} 1 & 0 \\ 1 & 1 \end{pmatrix}, 
\begin{pmatrix} 0 & 1 \\ 1 & 0 \end{pmatrix} \right\} \right).
\end{equation}
(To see that the right-hand expression defines a subgroup of $GL_2(\mbz/4\mbz)$, note that 
\[
\left\{ \begin{pmatrix} 0 & 0 \\ 0 & 0 \end{pmatrix}, \begin{pmatrix} 1 & 1 \\ 0 & 1 \end{pmatrix}, 
\begin{pmatrix} 1 & 0 \\ 1 & 1 \end{pmatrix}, \begin{pmatrix} 0 & 1 \\ 1 & 0 \end{pmatrix} \right\} \subseteq
M_{2\times 2}(\mbz/2\mbz)
\]
is closed under addition and under $GL_2(\mbz/2\mbz)$-conjugation.)

Since $\gal(\mbq(E[4])/\mbq)$ is a proper subgroup of $GL_2(\mbz/4\mbz)$ (even though 
$\gal(\mbq(E[2])/\mbq) = GL_2(\mbz/2\mbz)$) one has $4 \mid m_E$, and in this case the restriction map 
$\gal(\mbq(E[m_E])/\mbq) \twoheadrightarrow \gal(\mbq(E[4])/\mbq)$ induces a graph morphism
\begin{equation} \label{graphmorphism}
\mc{G}_E = \mc{G}_E(m_E) \twoheadrightarrow \mc{G}_E(4),
\end{equation}
which is surjective in the sense that it carries the vertex set $\mc{V}(m_E)$ onto $\mc{V}(4)$ and likewise 
carries $\mc{E}(m_E)$ onto $\mc{E}(4)$.

On the other hand, using \eqref{galoisatlevel4}, one finds that the directed graph $\mc{G}_E(4)$ is as follows. \\

\begin{equation} \label{graphGE}
\begin{pgfpicture}{0cm}{0cm}{10cm}{0.5cm}
\pgfnodecircle{Node1}[fill]{\pgfxy(1,.5)}{0.1cm}
\pgfnodecircle{Node2}[fill]{\pgfxy(4,.5)}{0.1cm}
\pgfnodecircle{Node3}[fill]{\pgfxy(6, .5)}{0.1cm}
\pgfnodecircle{Node4}[fill]{\pgfxy(8, .5)}{0.1cm}
\pgfnodeconnline{Node3}{Node2}
\pgfnodeconnline{Node3}{Node4}
\pgfsetendarrow{\pgfarrowtriangle{3pt}}
\pgfxyline(6,.5)(5,.5)
\pgfsetendarrow{\pgfarrowtriangle{3pt}}
\pgfxyline(6,.5)(7,.5)
\pgfputat{\pgfxy(.5, .1)}{\pgfbox[left,center]{$(2, 1)$}}
\pgfputat{\pgfxy(3.5,.1)}{\pgfbox[left,center]{$(2, -1)$}}
\pgfputat{\pgfxy(5.5,.1)}{\pgfbox[left,center]{$(-1, 1)$}}
\pgfputat{\pgfxy(7.5,.1)}{\pgfbox[left,center]{$(0, -1)$}}
\end{pgfpicture}
\end{equation}

\medskip

\subsubsection{Infinitely many primes $p$ for which $|E(\mbf_p)|$ is prime}

The non-CM case of a conjecture of Koblitz (see \cite{koblitz} and also \cite{zywina}) expresses (in our 
terminology) that for any non-CM elliptic curve $E$, the existence of a single directed edge in $\mc{G}_E$ implies 
the existence of infinitely many primes $p$ for which $| E(\mbf_p) |$ is prime.  Taking $E$ to be the elliptic 
curve given by \eqref{counterexampleE}, we see by the surjectivity of \eqref{graphmorphism} together with 
\eqref{graphGE} that $\mc{G}_E$ contains at least one directed edge.  Thus, assuming Koblitz's conjecture, there 
are infinitely many primes $p$ for which $| E(\mbf_p) |$ is prime.

\subsubsection{Finitely many aliquot cycles for $E$}

Continuing with the example \eqref{counterexampleE}, by the surjectivity of \eqref{graphmorphism} together with 
\eqref{graphGE}, we see that $\mc{G}_E$ contains no closed walks at all.  By Proposition \ref{equivproposition}, 
there are only finitely many aliquot cycles $(p_1,p_2, \dots, p_L)$ for $E$.  This particular example may be 
explained as follows. 
Whenever $p_2 = | E(\mbf_{p_1}) |$ for some prime $p_1$, we see from \eqref{graphGE} that 
$(\tr (\text{Frob}_{\mbq(E[4])}(p_1)), \det (\text{Frob}_{\mbq(E[4])}(p_1)) )= (-1,1)$ (otherwise, 
$|E(\mbf_{p_1})|$ would be even).  But then 
$(\tr (\text{Frob}_{\mbq(E[4])}(p_2)), \det (\text{Frob}_{\mbq(E[4])}(p_2)) ) \in \{ (0,-1),(2,-1) \}$, in which case 
$| E(\mbf_{p_2}) |$ must be even.  One deduces that $E$ has no aliquot cycles of length $L \geq 2$, and indeed no 
aliquot \emph{sequences} of length $L \geq 3$.

\begin{remark}
There is a modular curve $X$ of level $4$ and genus $0$ with $| X(\mbq) | = \infty$, whose non-cuspidal 
$\mbq$-rational points correspond to elliptic curves $E'$ for which $-\gD_{E'}$ is a perfect square. 
For almost all such elliptic curves $E'$, one may find an appropriate twist $E$ of $E'$ for which 
\eqref{galoisatlevel4} holds, and thus for which $\ds \lim_{x \rightarrow \infty} \pi_{E,L}(x) < \infty$.  The 
elliptic curve \eqref{counterexampleE} is one such example.
\end{remark}

\section{Acknowledgments}

The author gratefully acknowledges A.C. Cojocaru, who first brought this question to my attention, and also J. 
Silverman for a stimulating discussion.  Also many thanks to A. Sutherland, who provided help with the computations (a description of the software used therein may be found in \cite{kedyalasutherland}).  Finally, thanks to the anonymous referee for a careful reading of the manuscript and several helpful comments.

\newpage

\section{Appendix:  explicit lists of aliquot cycles}

The following tables list explicitly the normalized aliquot cycles with $p_1 \leq 10^{13}$ for each elliptic curve in 
Tables 2 and 3 of Section \ref{numerics}.  As mentioned before, the list for the first and third elliptic curves already appear in the 
literature (see \cite{silvermanstange}).

\begin{equation*} 
\begin{array}{|ccc|}  
\hline & & \\
& E :  y^2 +y = x^3 - x, \; L = 2 &  \\ 
& & \\
\hline 
& & \\
(1622311,1622471) &
(209051131,209065277) &
(435197207,435203627) \\
(1039959127,1040001691) &
(1129509221,1129533787) &
(1226864057,1226882263) \\
(2352481871,2352558343) &
(2611684883,2611740823) &
(2948995759,2949055441) \\
(3694651133,3694724861) &
(3700382359,3700422013) &
(5683068649,5683159501) \\
\hline
\end{array}
\end{equation*}

\newpage

\begin{equation*} 
\begin{array}{|ccc|}  
\hline & & \\
& E :  y^2 +y = x^3 - x, \; L = 2 &  \\ 
& \text{(continued)} & \\
& & \\
\hline 
& & \\
(882103493123,882104446687) &
(904777224133,904777406573) &
(931427616797,931428699827) \\
(941072287627,941072777989) &
(947754240637,947755837411) &
(975303777571,975304270909) \\
(988356964733,988357609933) & 
(1018492461589,1018492673593) &
(1022431333663,1022432126333) \\
\hline
\end{array}
\end{equation*}

\newpage

\begin{equation*} 
\begin{array}{|ccc|}  
\hline & & \\
& E :  y^2 +y = x^3 - x, \; L = 2 &  \\ 
& \text{(continued)} & \\
& & \\
\hline 
& & \\
(4750616484799,4750617009461) &
(4864298430709,4864299633001) &
(4998935089639,4998935423441) \\
(5040113449321,5040115562761) &
(5046154404541,5046155923399) &
(5055133170851,5055134417677) \\
(5093618078491,5093620358693) &
(5115989050199,5115991682531) &
(5163532995193,5163535140559) \\
\hline 
\end{array}
\end{equation*}
\begin{equation*}
\begin{array}{|ccc|} 
\hline & & \\
& E :  y^2 +y = x^3 - x, \; L = 3 &  \\
& & \\
\hline
& & \\
& (667166677657,667167059821,667168096487) & \\
& (1222850184103,1222851489439,1222851685021) & \\
& (2293567457947,2293568565859,2293567864171) & \\
& & \\
\hline
\end{array}
\end{equation*}

\begin{equation*} 
\begin{array}{|ccc|}  
\hline & & \\
& E :  y^2 = x^3 + 6 x - 2, \; L = 2 &  \\ 
& & \\
\hline 
& & \\
(1548181,1549957) &
(8418001,8420869) &
(27020971,27023203) \\
(41099887,41102779) &
(55475983,55485487) &
(103188703,103189183) \\
(103560409,103562257) &
(247178983,247205683) &
(311333227,311334547) \\
(313230349,313253617) &
(356804113,356827567) &
(422576281,422601397) \\
\hline
\end{array}
\end{equation*}

\newpage

\begin{equation*} 
\begin{array}{|ccc|}  
\hline & & \\
& E :  y^2 = x^3 + 6 x - 2, \; L = 2 &  \\ 
& \text{(continued)} & \\
& & \\
\hline 
& & \\
(286455303427,286455801883) &
(290995669561,290996513623) &
(295016767207,295017062383) \\
(304611562393,304611905011) &
(316426926331,316427396851) & 
(320194228441,320194439881) \\
(323392992001,323393683099) &
(323759613889,323759796799) &
(338730801697,338731733731) \\
\hline
\end{array}
\end{equation*}

\newpage

\begin{equation*} 
\begin{array}{|ccc|}  
\hline & & \\
& E :  y^2 = x^3 + 6 x - 2, \; L = 2 &  \\ 
& \text{(continued)} & \\
& & \\
\hline 
& & \\
(1471955283163,1471956405817) &
(1502699448073,1502701181323) &
(1507772573563,1507772756107) \\
(1534321718893,1534322621389) &
(1543059848767,1543060189549) &
(1557709419001,1557710634481) \\
(1596321026737,1596321153241) &
(1600288917499,1600288999747) &
(1602551186461,1602551703271) \\
\hline
\end{array}
\end{equation*}

\newpage

\begin{equation*} 
\begin{array}{|ccc|}  
\hline & & \\
& E :  y^2 = x^3 + 6 x - 2, \; L = 2 &  \\ 
& \text{(continued)} & \\
& & \\
\hline 
& & \\
(3939289466029,3939290731447) &
(3957221258971,3957223536367) &
(3962663691739,3962665018933) \\
(3966231147841,3966232793407) &
(4002034502353,4002036282253) &
(4092817954969,4092818340307) \\
(4135972471621,4135973009227) &
(4140190711471,4140192718999) &
(4146879215407,4146881516893) \\
\hline
\end{array}
\end{equation*}

\newpage

\begin{equation*} 
\begin{array}{|ccc|}  
\hline & & \\
& E :  y^2 = x^3 + 6 x - 2, \; L = 2 &  \\ 
& \text{(continued)} & \\
& & \\
\hline 
& & \\
(7596384984733,7596386908723) &
(7601076948943,7601079053443) &
(7679016909451,7679018357047) \\
(7734864764353,7734868937227) &
(7803721174981,7803721746607) &
(7869840059293,7869841521649) \\
(7908295643761,7908297236863) &
(7917659324917,7917663672379) &
(7929381283273,7929381530593) \\
\hline 
\end{array}
\end{equation*}
\begin{equation*}
\begin{array}{|cc|} 
\hline & \\
\quad\quad\quad\quad\quad\quad\quad\quad\quad\quad\quad\quad E :  y^2 = x^3 + 6 x - 2, & L = 3 \quad\quad\quad\quad\quad\quad\quad\quad\quad\quad\quad\quad\quad \\
& \\
\hline
& \\
(3513403471,3513471439,3513464893) &
(5917646737,5917698049,5917764211) \\
(9297148129,9297305233,9297275893) &
(139615969579,139616705689,139616623969) \\
(189978195889,189978788803,189978492391) &
(569053547773,569053692577,569054881243) \\
(1831670030833,1831670497897,1831672166089) &
(3248402980183,3248405169091,3248404591687) \\
(6298775494003,6298776702157,6298778026447) &
(6388976440057,6388980909931,6388979773633) \\
(7686931242769,7686932123911,7686935322373) &
(8467816169881,8467816902373,8467820440273) \\
& \\
\hline
\end{array}
\end{equation*}

\begin{equation*}
\begin{array}{|ccc|} 
\hline & & \\
& E :  y^2 = x^3 + 6 x - 2, \; L = 4 & \\
& & \\
\hline
& & \\
& (2492223637,2492224957,2492229163,2492239501) & \\
& & \\
\hline
\end{array}
\end{equation*}

\newpage

\begin{equation*} 
\begin{array}{|ccc|}  
\hline & & \\
& E :  y^2 +y = x^3 + x^2, \; L = 2 &  \\ 
& & \\
\hline 
& & \\
(853,883) &
(77761,77999) &
(1147339,1148359) \\
(1447429,1447561) &
(82459561,82471789) &
(109165543,109180121) \\
(253185307,253194619) &
(320064601,320079131) &
(794563993,794571803) \\
(797046407,797057473) &
(2185447367,2185504261) &
(2382994403,2383029443) \\
\hline
\end{array}
\end{equation*}

\newpage

\begin{equation*} 
\begin{array}{|ccc|}  
\hline & & \\
& E :  y^2 +y = x^3 + x^2, \; L = 2 &  \\ 
& \text{(continued)} & \\
& & \\
\hline 
& & \\
(1005332211467,1005333480829) &
(1313726076797,1313726385929) &
(1776998505169,1776998717597) \\
(1401963188881,1401964133147) &
(1468421896367,1468423289051) &
(1546682234947,1546682276779) \\
(1938065045771,1938067221319) &
(1989557561137,1989559104199) &
(1078711598329,1078712343781) \\
\hline
\end{array}
\end{equation*}

\newpage

\begin{equation*} 
\begin{array}{|ccc|}  
\hline & & \\
& E :  y^2 +y = x^3 + x^2, \; L = 2 &  \\ 
& \text{(continued)} & \\
& & \\
\hline 
& & \\
(7206978649363,7206981287353) &
(7652454700853,7652454845663) &
(4196987079577,4196987186443) \\
(4229935882369,4229936473291) &
(4938494358451,4938497989609) &
(7705465706023,7705470476959) \\
(9007603463153,9007607774507) &
(7520264498347,7520265119509) &
(7474974624707,7474975741241) \\
\hline 
\end{array}
\end{equation*}
\begin{equation*}
\begin{array}{|ccc|} 
\hline & & \\
& E :   y^2 +y = x^3 + x^2, \; L = 3 &  \\
& & \\
\hline
& & \\
& (658501858783,658502719313,658502576161) & \\
& (9690030049321,9690032636563,9690032993611) & \\
& & \\
\hline
\end{array}
\end{equation*}

\newpage

\begin{equation*} 
\begin{array}{|ccc|}  
\hline & & \\
& E :  y^2 + xy + y = x^3 - x^2, \; L = 2 &  \\ 
& & \\
\hline 
& & \\
(15782639,15784843) &
(190661353,190664659) &
(502321091,502327927) \\
(623231569,623231993) &
(848089241,848132891) &
(867592309,867624829) \\
(3416538269,3416597377) &
(3717074213,3717173309) &
(4238113591,4238209777) \\
(5152594561,5152642949) &
(6089286341,6089340407) &
(9570960601,9571090813) \\
\hline
\end{array}
\end{equation*}

\newpage

\begin{equation*} 
\begin{array}{|ccc|}  
\hline & & \\
& E :  y^2 + xy + y = x^3 - x^2, \; L = 2 &  \\ 
& \text{(continued)} & \\
& & \\
\hline 
& & \\
(1221199212977,1221199489259) &
(1225209348247,1225210933117) &
(1243835349979,1243836873739) \\
(1255914949763,1255916211031) &
(1369053007537,1369053353329) &
(1370852196319,1370852518141) \\
(1410180156367,1410180597691) &
(1496405539043,1496407012183) &
(1503579173743,1503580816483) \\
\hline
\end{array}
\end{equation*}

\newpage

\begin{equation*} 
\begin{array}{|ccc|}  
\hline & & \\
& E :  y^2 + xy + y = x^3 - x^2, \; L = 2 &  \\ 
& \text{(continued)} & \\
& & \\
\hline 
& & \\
(5177649685927,5177652271321) &
(5339612210879,5339615353531) &
(5422441462301,5422441615883) \\
(5464016097109,5464019597929) &
(5602425369113,5602427266333) &
(5679530846813,5679533819623) \\
(5794494649247,5794494840319) &
(5838118509373,5838120939103) &
(5852477074669,5852477985251) \\
\hline 
\end{array}
\end{equation*}
\begin{equation*}
\begin{array}{|ccc|} 
\hline & & \\
& E :   y^2 + xy + y = x^3 - x^2, \; L = 3 &  \\
& & \\
\hline
& & \\
& (34792873363,34793192633,34792965611) & \\
& (253482826411,253483159049,253483332217) & \\
& (279004675277,279004689461,279005168759) & \\
& (873362421247,873362444281,873364133731) & \\
& & \\
\hline
\end{array}
\end{equation*}

\end{document}